\documentclass[11pt]{amsart}
\usepackage[left= 3.9  cm,right= 3.9 cm,top= 3.0 cm,bottom= 3.0 cm]{geometry}
\usepackage[english]{babel}    
\usepackage{amssymb, amsmath} 
\usepackage{amsthm}
\usepackage{xcolor}
\usepackage{stmaryrd}
\usepackage{graphicx} 
\usepackage[arrow,curve,matrix,tips,2cell]{xy}
\usepackage{mathrsfs}
\usepackage[capitalise]{cleveref}
\usepackage{microtype}
\usepackage[alphabetic]{amsrefs}

\begin{document}

\setlength{\parindent}{0pt}

\theoremstyle{plain}

\newtheorem{Thm}{Theorem}
\newtheorem{Lem}[Thm]{Lemma}
\newtheorem{LemDef}[Thm]{Lemma/Definition}
\newtheorem{Cor}[Thm]{Corollary}
\newtheorem{Pro}[Thm]{Proposition}
\newtheorem*{nono-Thm}{Theorem}
\newtheorem*{ThmA}{Theorem A}
\newtheorem*{ThmB}{Theorem B}
\newtheorem*{ThmC}{Theorem C}
\newtheorem*{ThmD}{Theorem D}
\newtheorem*{nono-Lem}{Lemma}
\newtheorem*{nono-Cor}{Corollary}
\newtheorem*{nono-Pro}{Proposition}

\theoremstyle{definition}

\newtheorem{Def}[Thm]{Definition}
\newtheorem{Exe}[Thm]{Exercise}
\newtheorem{Exa}[Thm]{Example}
\newtheorem{Rem}[Thm]{Remark}
\newtheorem{Fac}[Thm]{Fact}
\newtheorem{Que}[Thm]{Question}
\newtheorem{Con}[Thm]{Conjecture}

\newcommand{\CA}{\mathcal A}
\newcommand{\CB}{\mathcal B}
\newcommand{\CC}{\mathcal C}
\newcommand{\CD}{\mathcal D}
\newcommand{\CE}{\mathcal E}
\newcommand{\CF}{\mathcal F}
\newcommand{\CG}{\mathcal G}
\newcommand{\CH}{\mathcal H}
\newcommand{\CI}{\mathcal I}
\newcommand{\CJ}{\mathcal J}
\newcommand{\CK}{\mathcal K}
\newcommand{\CL}{\mathcal L}
\newcommand{\CM}{\mathcal M}
\newcommand{\CN}{\mathcal N}
\newcommand{\CO}{\mathcal O}
\newcommand{\CQ}{\mathcal Q}
\newcommand{\CR}{\mathcal R}
\newcommand{\CS}{\mathcal S}
\newcommand{\CT}{\mathcal T}
\newcommand{\CU}{\mathcal U}
\newcommand{\CV}{\mathcal V}
\newcommand{\CW}{\mathcal W}
\newcommand{\CX}{\mathcal X}
\newcommand{\CY}{\mathcal Y}
\newcommand{\CZ}{\mathcal Z}

\newcommand{\MA}{\mathbb A}
\newcommand{\MB}{\mathbb B}
\newcommand{\MC}{\mathbb C}
\newcommand{\MD}{\mathbb D}
\newcommand{\ME}{\mathbb E}
\newcommand{\MF}{\mathbb F}
\newcommand{\MG}{\mathbb G}
\newcommand{\MH}{\mathbb H}
\newcommand{\MI}{\mathbb I}
\newcommand{\MJ}{\mathbb J}
\newcommand{\MK}{\mathbb K}
\newcommand{\ML}{\mathbb L}
\newcommand{\MM}{\mathbb M}
\newcommand{\MN}{\mathbb N}
\newcommand{\MO}{\mathbb O}
\newcommand{\MP}{\mathbb P}
\newcommand{\MQ}{\mathbb Q}
\newcommand{\Reals}{\mathbb R}
\newcommand{\MS}{\mathbb S}
\newcommand{\MU}{\mathbb U}
\newcommand{\MV}{\mathbb V}
\newcommand{\MW}{\mathbb W}
\newcommand{\MX}{\mathbb X}
\newcommand{\MY}{\mathbb Y}
\newcommand{\MZ}{\mathbb Z}

\newcommand{\StA}{\mathscr A}
\newcommand{\CP}{\mathbb C P^{\infty}}

\newcommand{\ab}{\text{ab}}
\newcommand{\im}{\text{im}}
\newcommand{\ch}{\text{ch}}
\newcommand{\rk}{\text{rk}}
\newcommand{\st}{\text{st}}
\newcommand{\Sq}{\text{Sq}}
\newcommand{\colim}{\text{colim}}
\newcommand{\hofib}{\text{hofib}}

\newcommand{\rR}{\rightarrow}
\newcommand{\RR}{\Rightarrow}
\newcommand{\LR}{\Leftrightarrow}
\newcommand{\id}{\text{id}}
\newcommand{\Hom}{\text{Hom}}
\newcommand{\Ext}{\text{Ext}}
\newcommand{\Tor}{\text{Tor}}
\newcommand{\pt}{\star}
\newcommand{\op}{\text{op}}
\newcommand{\ev}{\text{ev}}
\newcommand{\pr}{\text{pr}}
\newcommand{\even}{\text{even}}
\newcommand{\odd}{\text{odd}}
\newcommand{\const}{\text{const}}
\newcommand{\Gal}{\text{Gal}}
\newcommand{\maps}{\text{maps}}
\newcommand{\map}{\text{map}}
\newcommand{\mog}{\text{mog}}
\newcommand{\Aut}{\text{Aut}}
\newcommand{\Conf}{\text{Conf}}
\newcommand{\Sing}{\text{Sing}}
\newcommand{\height}{\text{height}}
\newcommand{\Fr}{\text{Fr}}
\newcommand{\GL}{\text{GL}}
\newcommand{\St}{\text{St}}
\newcommand{\Sym}{\text{Sym}}
\newcommand{\TOP}{\text{TOP}}
\newcommand{\APL}{\left(A_\text{PL}\right)}
\newcommand{\CPL}{\left(C_\text{PL}\right)}
\newcommand{\Cpl}{C_\text{PL}}
\newcommand{\Apl}{A_\text{PL}}
\newcommand{\CDGA}{\text{CDGA}}
\newcommand{\character}{\text{char}}
\newcommand{\Fadd}{F_{\text{add}}}
\newcommand{\Fell}{F_{\text{ell}}}
\newcommand{\Fmult}{F_{\text{mult}}}
\newcommand{\Funiv}{F_{\text{univ}}}
\newcommand{\MT}{\text{\bf MT}}
\newcommand{\hAut}{\text{hAut}}
\newcommand{\Diff}{\text{Diff}}
\newcommand{\DiffPlus}{\text{Diff}^{\ +}}
\newcommand{\Gr}{\text{Gr}}
\newcommand{\HE}{H_{S^1}}
\newcommand{\EE}{H_{S^1,\text{loc}}}
\newcommand{\Spin}{\text{Spin}}
\newcommand{\MSpin}{\text{\bf MSpin}}
\newcommand{\be}{b_{\text{even}}}
\newcommand{\bo}{b_{\text{odd}}}

\newcommand{\commentJR}[1]{\textcolor{blue}{#1}}

\title[Tautological classes and smooth bundles over BSU(2)]{Tautological classes and smooth bundles \\ over BSU(2)}

\author{Jens Reinhold}
\address{Department of Mathematics, Stanford University, Stanford, California 94305}
\email{\texttt{jreinh@stanford.edu}}
\date{}
\urladdr{\url{https://sites.google.com/stanford.edu/jreinhold}}
\thanks{Jens Reinhold is supported by the E. K. Potter Stanford Graduate Fellowship}

\begin{abstract}
For a Lie group $G$ and a smooth manifold $W$, we study the difference between smooth actions of $G$ on $W$ and fiber bundles over the classifying space of $G$ with fiber $W$ and structure group $\Diff(W)$. In particular, we exhibit smooth manifold bundles over $BSU(2)$ that are not induced by an action. The main tool for reaching this goal is a technical result that gives a constraint for the values of tautological classes of the fiber bundle associated to a group action. 
\end{abstract}

\setcounter{secnumdepth}{2}

\maketitle{}

\section{{Introduction and statement of results}} 
\normalsize
Throughout this work, $G$ stands for a Lie group and $W$ denotes a closed connected smooth oriented manifold of dimension $2n$. If $G$ acts smoothly on $W$, this gives rise to a map $BG \to B\Diff(W)$, where $\Diff(W)$ stands for the group of orientation-preserving diffeomorphisms of $W$ endowed with the Whitney $C^{\infty}$-topology. Such maps classify fiber bundles with fiber $W$ and structure group $\Diff(W)$, which will henceforth be called \emph{smooth $W$-bundles}. Let us denote the set of all smooth actions of $G$ on $W$ by $\Hom (G,\Diff(W))$. For a paracompact model of the classifying space $BG$, the above assignment can be seen as a map
\begin{equation*}\label{Bmap}
 \Hom (G,\Diff(W)) / \Diff(W) \xrightarrow{B}  \{\text{Smooth} \ W\text{-bundles over} \ BG\} /\! \sim,
\end{equation*}
where $\sim$ stands for the equivalence relation of bundle isomorphism. Let us now consider the following dichotomy.

\begin{Def} A smooth $W$-bundle over $BG$ is called \emph{kinetic} if it lies in the image of $B$, and \emph{non-kinetic} otherwise. \end{Def}

This leads to the following question.

\begin{Que}\label{question} For which pairs $(G,W)$ do there exist non-kinetic smooth $W$-bundles over $BG$?
\end{Que}

In the case that $G$ is a discrete group, Question \ref{question} relates to an open problem about flat manifold bundles. Indeed, for discrete $G$, a smooth $W$-bundle $E \to BG$ is kinetic if and only if it admits a flat structure: both properties are equivalent to the condition that the classifying map $BG \to B\Diff(W)$ admits a lift along $B\Diff(W)^{\delta} \to B\Diff(W)$, where $B\Diff(W)^{\delta}$ is the classifying space of the group $\Diff(W)$ with the discrete topology. Thus, in the case that $G$ is discrete, a non-kinetic bundle $BG \to B\Diff(W)$ is one that does not admit any flat structure. For instance, suppose that $G = \pi_1(\Sigma_a)$ and $W = \Sigma_b$, where $\Sigma_g$ is a closed connected surface of genus $g$. Establishing the existence of a non-kinetic bundle in this case is equivalent to showing that there is a surface bundle over a surface that does not admit a flat structure. Whether such bundles exist is an open problem, see e.g.~the question at the end of the introduction of \cite{Church}.

In this work, we instead focus on the case where $G$ is a compact connected Lie group. For most purposes, we in fact specialize to the case $G = SU(2)$. We use characteristic classes of manifold bundles, i.e., cohomology classes $\alpha \in H^{\ast}(B\Diff(W);\MQ)$, as a tool to answer question \ref{question} in infinitely many cases. For any such $\alpha$ and any smooth action of $G$ on $W$, we can pull $\alpha$ back along the induced map $s\colon BG \to B\Diff(W)$ to get a cohomology class $s^{\ast} \alpha \in H^{\ast}(BG;\MQ)$. This process can be seen as a map
\begin{equation*}
D^{G,W}_{\alpha} \colon \Hom(G,\Diff(W))/\Diff(W) \to H^{\ast}(BG;\MQ), 
\end{equation*}
and one of the results of this paper is a statement about the image of this map for $G = SU(2)$ and $\alpha$ particular generalized Miller--Morita--Mumford classes. We recall the definition of these classes in Section \ref{MMM}. Our results are concerned with manifolds that have the following property.

\begin{Def}\label{classesofmfds}
A manifold $W$ of dimension $2n$ is called \emph{rationally odd}, if all its rational cohomology apart from $H^0(W;\MQ)$ and $H^{2n}(W;\MQ)$ is supported in odd degrees.
\end{Def}

Note that any manifold $W$ of dimension $2n$ which is rationally odd satisfies $\chi(W) \leq 2$, and $\chi(W) = 2$ if and only if $W$ has the rational cohomology of $S^{2n}$. Examples of manifolds that are rationally odd are given by $W_g = \#^g (S^n \times S^n)$ for $n \geq 1$ odd and $g \geq 0$ arbitrary, where $\#^g$ denotes $g$-fold connected sum. The following statement is the first main result of this paper.

\begin{ThmA}\label{mainthm} 
Let $G = SU(2)$ and let $W$ be a manifold of dimension $2n$. Suppose that $W$ is rationally odd and satisfies $\chi(W) < 0$. For any non-trivial smooth action  of $SU(2)$ on $W$ that induces the map 
$s\colon BSU(2) \to B\emph{\Diff(W)},$ let $b_1, \dots, b_n \in \MQ$ be such that the formula
\begin{equation*} 
s^{\ast} \kappa_{ep_i} = b_i\chi(W)c_2^{i} \in H^{4i}(BSU(2);\MQ)
\end{equation*}
holds, where $c_2 \in H^4(BSU(2);\MQ)$ denotes the second Chern class, $p_i$ denotes the $i^{\text{th}}$ Pontryagin class and $e$ denotes the Euler class. Then, the numbers $b_i$ are all integers and moreover, their greatest common divisor is a power of 2. 
\end{ThmA}

\begin{Rem} 
As any power of $2$ is non-zero, one obtains from Theorem A that for any manifold as stated and \emph{any} non-trivial action of $SU(2)$ on $W$ that induces
\begin{equation*}
s\colon BSU(2) \to B\Diff(W),
\end{equation*}
at least one of the classes $s^{\ast}\kappa_{ep_i} \in H^{4i}(BSU(2);\MQ)$ is non-zero.
\end{Rem}

It is worth mentioning that for $G = SU(2)$ and a fixed manifold $W$, the individual classes $\kappa_{ep_i}$ can in fact take infinitely many values when all possible actions of $G$ on $W$ are considered.  We prove this in Section \ref{infinitely}. It is also explained how this implies that there exist infinitely many non-isomorphic kinetic smooth $S^2 \times S^2$-bundles over $BSU(2)$, see Corollary \ref{InfinitelyMany}. Once Theorem A is established, we use it to prove the following  second main result. It answers Question \ref{question} in infinitely many cases, in particular it applies to $W_g = \#^g (S^n \times S^n)$ if $n \geq 3$ is odd and $g > 1$.

\begin{ThmB}\label{ThmB} Let $W$ be a closed connected manifold which is rationally odd, admits a non-trivial smooth action of $SU(2)$ and satisfies $\chi(W) < 0$. Then there exist smooth $W$-bundles over $BSU(2)$ which are non-kinetic.\end{ThmB}

To our knowledge, Question \ref{question} has not been addressed at any point in the literature; Theorem B seems to be the first existence result on manifold bundles that are non-kinetic. If one replaces $\Diff(W)$ by a compact Lie group $H$, i.e., if one considers $H$-principal bundles instead of smooth $W$-bundles, the question of understanding the map $\Hom(G,H) \to [BG,BH]$ and especially its target, is a classical problem in algebraic topology \cite{AM76, DwyerZabrodsky, Notbohm, Bob2}. 

\subsection*{Acknowledgements} 
I am deeply grateful to my PhD advisor S{\o}ren Galatius for his continuous encouragement and support and I would like to thank Bertram Arnold and Alexander Kupers, who also suggested the name (non-)kinetic for the two types of bundles that are considered, for helpful discussions.  

A substantial part of the research outlined in this work was done during two very pleasurable visits at the Department of Mathematics at the University of Copenhagen, which were financially supported by the Danish National Research Foundation through the Centre for Symmetry and Deformation (DNRF92) and by the European Research Council (ERC) under the European Union's Horizon 2020 research and innovation program (grant agreement No 682922).

\section{Recollections on the generalized Miller--Morita--Mumford-classes}\label{MMM}
As above, let $W$ be a closed connected manifold of dimension $2n$. Let us recall the definition of generalized Miller-Morita-Mumford-classes $\kappa_c \in H^{\ast}(B\Diff(W);\MQ)$ \cite{Miller, Morita, Mumford}. For a smooth $W$-bundle $\pi\colon E \to B$, let $P \to B$ denote the associated principal $\Diff(W)$-bundle. The vertical tangent bundle $T_{\pi}E$ is the vector bundle  
$P \times_{\Diff(W)} TW \to P \times_{\Diff(W)} W = E$. For any class $c \in H^{\ast}(BSO(2n);\MQ) \cong \MQ[p_1,p_2,\dots,p_{n-1},p_n,e]/(e^2-p_n)$, one may form 
\[\kappa_{c} := \int_{\pi} c(T_{\pi}E) \in H^{\ast-2n}(B;\MQ),\]
where the integral stands for integration along the fibers. In the case that $E$ and $B$ are smooth manifolds, the two constructions involved in this definition have a geometric interpretation: the vertical tangent bundle $T_{\pi}E$ is isomorphic to $\ker \left(d\pi \colon TE \to TB\right)$, and the map $\int_{\pi} \colon H^{\ast}(E;\MQ) \to H^{\ast-2n}(B;\MQ)$ is Poincar{\'e} dual to the map $\pi_{\ast}$ induced by $\pi$ on homology.

The corresponding class $\kappa_c$ is usually called the \emph{tautological class} or \emph{generalized Miller-Morita-Mumford-class} associated to $c$, though sometimes it is also simply referred to by the name $\kappa$-\emph{class}. This construction is natural in $B$, so it makes sense to regard the $\kappa_c$ as classes in the cohomology of $B\Diff(W)$, the base of the universal smooth $W$-bundle. We usually require $c$ to be a monomial in the classes $p_1,\dots,p_{n-1}$ and $e$; the collection of all such monomials is denoted by $\CB$.
\section{Circle actions}
Suppose the circle $S^1$ acts smoothly on a manifold $W$ of dimension $2n$, giving rise to a map 
$s\colon BS^1 \to B\Diff(W).$ For any class $c \in H^{\ast}(BSO(2n);\MQ)$, it is useful to know the value of $s^{\ast}\kappa_{ec} \in H^{\ast}(BS^1;\MQ)$. Let us denote the canonical generator (i.e., the Euler class) of the $\MQ$-algebra $H^{\ast}(BS^1;\MQ)$ by $\gamma$. To give a formula for $s^{\ast}\kappa_{ec} \in H^{\ast}(BS^1;\MQ)$, let us first analyze the linear case, i.e., we consider a linear action of $S^1$ on $\mathbb R^{2n}$.  For any $a \in \MZ$, let us denote the $2$-dimensional real representation $S^1 \to S^1 = SO(2), z \mapsto z^a$ by $\ell_a$.
Then $\ell_0$ is the trivial 2-dimensional representation, and the $\ell_a$ for $a \neq 0$ are all the non-trivial irreducible real representations of $S^1$. Thus, any $2n$-dimensional real representation $V$ has the form $\ell_{a_1} \oplus \dots \oplus \ell_{a_n}$ for numbers $a_1, \dots, a_n \in \MZ$, the \emph{weights} of $V$. 

\begin{Lem} \label{Pontryagin}\label{Pontryagin} 
For $a_1, \dots, a_n \in \MZ$, the map $BS^1 \to BSO(2n)$ induced by the real representation $\ell_{a_1} \oplus \dots \oplus \ell_{a_n}$ sends $c \in \CB^{(2j)} \subset H^{2j}(BSO(2n);\MQ)$ to $\sigma_c(a_1,\dots,a_n)\gamma^{j} \in H^{2j}(BS^1;\MQ)$, which is defined as follows. For $c = p_i$, we set 
\begin{equation*}\label{cancel}
\sigma_{p_i}(a_1,\dots, a_n) = (-1)^i\sigma_{2i}(a_1,-a_1,a_2,-a_2,\dots,a_n,-a_n) 
= \sigma_{i}(a_1^2,a_2^2, \dots, a_n^2), 
\end{equation*}
where $\sigma_{j}$ denotes the $j^{\text{th}}$ elementary symmetric polynomial. We also set  
\begin{equation*}
\sigma_e(a_1,\dots,a_n) = a_1\cdots a_n.
\end{equation*}
Finally, we impose the conditions $\sigma_{cc'} = \sigma_c \sigma_{c'}$ to define $\sigma_c$ for all $c \in \CB$ (there is no contradiction arising from the relation $e^2 = p_n$).
\end{Lem}

\begin{Rem}
The second equality sign in equation (\ref{cancel}) above holds since the summands appearing in $\sigma_{2i}(a_1,-a_1,a_2,-a_2,\dots,a_n,-a_n)$ cancel each other, except those which contain both $a_k$ and $-a_k$ for $1 \leq k \leq n$ or neither of each of them. 
\end{Rem}

\begin{proof} Let $a \in \MZ$ be arbitrary and let us denote the complex representation $S^1 \to U(1), z \mapsto z^a$
by $L_a$. We then have $\ell_a \otimes \MC \cong L_{-a} \oplus L_a$. Thus, we obtain the following commutative diagram.
\begin{equation*}
\xymatrix{
S^1 \ar[r]^{\ell_{a_1} \oplus \dots \oplus \ell_{a_n}}  \ar[d]_{(L_{a_1},L_{-a_1},\dots,L_{a_n},L_{-a_{n}})}& SO(2n) \ar[d]^{\otimes \MC} \\
U(1) \times \dots \times U(1) \ar[r] & U(2n)
}
\end{equation*}
Going to classifying spaces and using the definition $p_i(V) = (-1)^ic_{2i}(V\otimes \MC)$ of the Pontryagin classes in terms of Chern classes, this diagram implies the assertion for all the Pontryagin classes. For the Euler class $e$, the assertion follows from the Whitney sum formula.
\end{proof}

We now prove a result about the possible fixed points that could arise from smooth circle actions on manifolds. For this, we consider the rational Betti numbers of a space $X$, given by $b_j(X) := \dim_{\MQ} H^j(X;\MQ)$. For a finite CW-complex $X$, we can define the following non-negative integers by summing up the even and odd Betti numbers, $\be (X) = \sum_{j = 0}^{\infty} b_{2j}(X)$ and $\bo (X) = \sum_{j = 1}^{\infty} b_{2j-1}(X)$. Note that the difference $\be (X) - \bo (X)$ is equal to the Euler characteristic $\chi(X)$ of $X$.

\begin{Lem}\label{circleactionbound} 
Suppose $S^1$ acts smoothly on the manifold $W$. Then, the set of fixed points $M = W^{S^1}$ is a submanifold of even codimension, and there is an integer $k \geq 0$ such that $\be (M) = \be(W) - k$ and $\bo (M) = \bo (W) - k$.
\end{Lem}
\begin{proof}
The first part of the statement is a well-known fact, see \cite[Chapter 2.6 and 2.7]{Duistermaat}. The idea of the proof of the second part, which was inspired by Example 3.4 in \cite{Oscarnew}, is to use equivariant localization. Recall that for any finite $S^1$-CW-complex $X$, one can consider the equivariant cohomology groups, defined as $\HE^{\ast}(X;\MQ) = H^{\ast}(X \times_{S^1} ES^1;\MQ).$ We then have $\HE^{\ast}(\pt;\MQ) = H^{\ast}(BS^1;\MQ) \cong \MQ[\gamma],$ and $\smash{\EE^{\ast}(X;\MQ)} := \HE^{\ast}(X;\MQ) \otimes_{\MQ[\gamma]} \MQ[\gamma^{\pm1}]$ is the localized equivariant cohomology of $X$. This $\MQ[\gamma ^{\pm 1}]$-module inherits a $\MZ/2$-grading coming from the even and odd degrees of $\HE^{\ast}(X;\MQ)$. By an induction over all cells with finite stabilizers, it is proved that the inclusion $\smash{X^{S^1} \hookrightarrow X}$ induces an isomorphism on $\EE^{\ast}(\ \_ \ ;\MQ)$, see \cite[XII.§3]{Borel}. Since $S^1$ acts trivially on $M$, we get that $\be(M)$ and $\bo(M)$ agree with the $\MQ[\gamma ^{\pm 1}]$-rank of the even and odd part of $\EE^{\ast}(W^{S^1}) =\EE^{\ast}(W)$. Now the Serre spectral sequence arising from the fibration $W \to W \times_{S^1} ES^1 \to BS^1$ shows that these numbers are bounded by $\be(W)$ and $\bo(W)$ and their difference agrees with the $\MQ[\gamma^{\pm 1}]$-Euler characteristic of $\EE^{\ast}(W^{S^1}) =\EE^{\ast}(W)$, which is equal to the Euler characteristic of $W$, since taking the Euler characteristic commutes with spectral sequences that have finitely many non-trivial differentials. Since $\chi(W) = \be(W) - \bo(W)$, this finishes the proof.
\end{proof}

Suppose $S^1$ acts smoothly on a $2n$-dimensional closed manifold $W$. Choose a Riemannian metric on $W$ which is invariant under the circle action. For any fixed point $x \in W$ of the action, we get a $2n$-dimensional orthogonal representation $S^1 \to O(T_xW)$ at $x$, which we denote by $V_x$. Since $V_x$ depends continuously on $x$, any two fixed points that can be joined by a path will produce isomorphic representations. The following result is central for proving Theorem A.
\begin{Lem}\label{pullbackformula} Suppose $S^1$ acts smoothly on a manifold $W$ of dimension $2n$. Let $W^{S^1} = M_1 \amalg \dots \amalg M_r$ be the decomposition of the fixed point set into connected components. Suppose that for $x \in M_i$, $1 \leq i \leq r$, the tangential representation $V_{x_i}$ has weights $a_{i,1}, \dots, a_{i,n}$. Then, for any $c \in \CB^{(2j)}$, the induced map \emph{$s\colon BS^1 \to B\Diff(W)$} satisfies $s^{\ast} \kappa_{ec} = \left(\sum_{i = 1}^r \chi(M_i)\sigma_c(a_{i,1},\dots,a_{i,n})\right)\gamma^{2j}$.
\end{Lem}
\begin{proof} For any fibration $E \to B$ for which the fiber $F$ is a finite CW-complex and the base $B$ is connected, Becker and Gottlieb define a transfer map of spectra $\tau\colon \Sigma^{\infty}_+ B \to \Sigma^{\infty}_+ E,$ see \cite{BeckerGottlieb, BrumfielMadsen}. This map is natural in the base and satisfies that the composition $\left(\Sigma^{\infty}_+ \pi \right) \circ \tau \colon \Sigma^{\infty}_+ B \to  \Sigma^{\infty}_+ B$ is multiplication by $\chi(F)$ on singular homology, see \cite{BeckerGottlieb}. Let us now consider the smooth $W$-bundle $\pi \colon E \to BS^1$ arising from pulling back the universal smooth $W$-bundle along the map $s$. The Becker-Gottlieb transfer $\tau\colon \Sigma^{\infty}_+ BS^1 \to \Sigma^{\infty}_+ E$ is then related to fiber integration as follows. For any characteristic class $c = c(T_{\pi} E) \in H^{2j}(E)$ of the vertical tangent bundle  of $\pi$, the equation $\tau^{\ast} c = \int_{\pi} ec = s^{\ast} \kappa_{ec}$ holds. We proceed by splitting $W$ into the manifolds $M_i$ and the open set $U = W - (M_1 \cup \dots \cup M_r)$. Let us decompose the smooth $W$-bundle $E \to BS^1$ accordingly. Moreover, let $V_i$ be pairwise disjoint regular open neighborhoods of the subspaces $M_i$ that are invariant under the $S^1$-action, for all $i = 1, \dots, r$, and let $V = V_1 \cup \dots \cup V_r$. We then have a pull-back square  of bundles over $BS^1$:
\begin{equation*}
\xymatrix{
E_{U \cap V} \ar[r] \ar[d]_j & E_U \ar[d]^h \\
E_V \ar[r]^k & E}
\end{equation*}
In the abelian group of homotopy classes of stable maps $\Sigma^{\infty}_+ BS^1 \to \Sigma^{\infty}_+ E$, $\tau_E = h_{\ast} \tau_{E_U} + k_{\ast} \tau_{E_V} - k_{\ast} j_{\ast} \tau_{E_{U \cap V}}$ holds, see \cite{BrumfielMadsen}. Furthermore, the circle actions on $E_U$ and $E_{U \cap V}$ do not have fixed points, which implies that the vertical tangent bundles for these manifold bundles are trivial. Finally, the total space of the bundle $E_V \to BS^1$ is trivial up to homotopy over $BS^1$, i.e., $E_V \simeq (M_1 \amalg \dots \amalg M_r) \times BS^1$. On the other hand, the vertical tangent bundle $T_{E_V}$ is not trivial: if we restrict it to $V_i$, we get that it is induced from the tangential representation $V_{x_i} \colon BS^1 \to BSO(2n)$. We can thus apply Lemma \ref{Pontryagin} to deduce the following equation.
\begin{equation*}
c(T_{E_V}) = \sum_{i = 1}^r c\left(T_{E_{V_{x_i}}}\right) = \left( \sum_{i = 1}^r c(V_{x_i}) 1_{M_i} \right) \times \gamma^{2j} = \sum_{i = 1}^r \sigma_c(a_{i,1},\dots,a_{i,n}) \left(1_{M_i} \times \gamma^{2j} \right)
\end{equation*}
For a trivial fibration with connected fiber $F$, the Becker-Gottlieb transfer simply multiplies by the Euler characteristic $\chi(F)$, so for the fibration $E_V \to BS^1$ we get by linearity that  $\tau^{\ast} (1_{M_i} \times \gamma^{2j}) = \chi(M_i)\gamma^{2j}.$ Combining everything, we thus obtain
\begin{eqnarray*}
s^{\ast} \kappa_{ec} &=& \tau_{E}^{\ast} c(T_E) \\ &=& \tau_{E_U}^{\ast} (c(T_{E_U})) +  \tau_{E_V}^{\ast} (c(T_{E_V})) - \tau_{E_{U \cap V}}^{\ast} (c(T_{E_{U \cap V}})) \\
&=& \tau_{E_V}^{\ast}(c(T_{E_V})) \\
&=& \sum_{i = 1}^r c(V_{x_i}) \tau_{E_V}^{\ast} (1_{M_i} \times \gamma^{2j}) = \left(\sum_{i = 1}^r \chi(M_i)\sigma_c(a_{i,1},\dots,a_{i,n})\right)\gamma^{2j},
\end{eqnarray*}
where we have used Lemma \ref{Pontryagin} in the last step. This finishes the proof.
\end{proof}

\section{Representation theory of SU(2)}
In this section, we recall some basic representation theory of $SU(2)$, see any book on the representation theory of Lie groups, for instance \cite{Broecker}. All finite-dimensional irreducible complex representations of $SU(2)$ are given by $V_{\lambda}$, where $2\lambda$ is a non-negative integer. The value $\lambda$ is sometimes called \emph{spin} of the representation. These representations are defined as $V_{\lambda} = \Sym^{2\lambda}(V_{\text{reg}}),$
where $V_{\text{reg}} = V_{1/2}$ is the regular representation $\MC^2$ on which $SU(2)$ acts by matrix multiplication. The weights of $V_{\lambda}$ are given by $-2\lambda, -2\lambda + 2, \dots, 2\lambda -2, 2\lambda$. It is also known which of these representations come from real representations: in \cite{Itzkowitz}, it is proven that for a positive integer $d$, there is no real irreducible $SU(2)$-representation of dimension $d$ if $d \equiv 2 \ (\text{mod }4)$. If $d$ is odd, there is exactly one irreducible real $SU(2)$-representation $V^d$ of dimension $d$, which has the property that $V^d \otimes_{\Reals} \MC \cong V_{{(d-1)}/{2}}$ is irreducible again. Finally, if $d$ is divisible by $4$, there is exactly one irreducible real representation $V^d$ of dimension $d$ and $V^d \otimes_{\Reals} \MC \cong V_{{d}/{4}-{1}/{2}} \oplus V_{{d}/{4}-{1}/{2}}$. As any real representation can be decomposed into irreducible representations, this implies the following result.

\begin{Pro}\label{weights} 
Let $V^d$ be any non-trivial $d$-dimensional real representation of $SU(2)$. Then the weights of $V^d \otimes_{\mathbb R} \MC$ are at most $d-1$ in absolute value, and there always exists a weight that is $\pm 1$ or $\pm 2$.
\end{Pro}

\section{Proof of Theorem A}

\begin{proof}
Assume $W$ is a closed, connected manifold of dimension $2n$. Recall the map   
\begin{equation*}
D = D^{SU(2),W}_{\kappa_{ep_i}} \colon \Hom(SU(2),\Diff(W))/\Diff(W) \to H^{4i}(BSU(2);\MQ),
\end{equation*}
from the introducton. It takes a smooth action of $SU(2)$ on $W$ that induces a map $s\colon BSU(2) \to B\Diff(W)$ 
and assigns the value $s^{\ast} \kappa_{ep_i} \in H^{4i}(BSU(2);\MQ)$ to it. Fix a maximal torus $S^1 \subset SU(2)$. For any smooth action of $SU(2)$ on $W$, we then get an induced action of the circle $S^1$ on $W$. The fixed point set $M = W^{S^{1}}$ of this action is a submanifold $M \subset W$ of even codimension, see Lemma \ref{circleactionbound}. By the same Lemma, we also deduce that $\be(M) \leq \be(W) = 2$ and $\chi(M) = \chi(W) < 0$. The first condition alone implies that $M$ is either $S^0$ or a connected manifold, and since $\chi(S^0) = 2$, the former cannot be the case. We have thus shown that $M$ has to be connected. For a point $x \in M$, the tangential $S^1$-representation $V_x$ at $x$ is given by a sequence $(a_1, \dots, a_{n})$ of weights. Since $M$ is connected, this sequence is independent of the choice of $x$. 

We now claim that at least one of these weights has to be $\pm 1$ or $\pm 2$. In the case that $x$ is fixed by the whole group $SU(2)$, the $S^1$-representation $V_x$ extends to a non-trivial representation of $SU(2)$, and this case follows from Proposition \ref{weights}. 

The second case is that $x$ is not a fixed point for the action of the whole group $SU(2)$. Then, the isotropy group of $x$ has to be a 1-dimensional subgroup of $SU(2)$ containing the chosen maximal torus $S^1$. There are only two such subgroups, namely $S^1$ and its normalizer $N \cong O(2)$. The orbit of $x$ is then either $S^2 = S^1 \backslash SU(2)$ or $\mathbb R P^2 = N \backslash SU(2)$. For both cases, we get that the weight of $V_x$ corresponding to the plane tangential to this orbit has to be equal to $1$. So we have seen that in all cases one of the weights has to be equal to $\pm1$ or $\pm 2$. 

We now study the values of $s^{\ast} \kappa_{ep_i} \in H^{4i}(BSU(2);\MQ)$. For this, let us recall that by Lemma \ref{pullbackformula}, we get that the composite map $t\colon BS^1 \to BSU(2) \to B\Diff(W)$ satisfies $t^{\ast} \kappa_{ep_i} = \chi(M) \sigma_i(a_1^2, \dots, a_n^2)\gamma^{2i} \in H^{4i}(BS^1;\MQ).$
Since the map on cohomology $H^{\ast}(BS^1;\MQ) \leftarrow H^{\ast}(BSU(2);\MQ)$ induced from the inclusion $S^1 \subset SU(2)$ is injective and sends $c_2$ to $\gamma^2$, we deduce that the map $s\colon BSU(2) \to B\Diff(W)$ satisfies $s^{\ast} \kappa_{ep_i} = \chi(W) b_ic_2^i \in H^{4i}(BSU(2);\MQ)$ with $b_i = \sigma_i(a_1^2, \dots, a_n^2)$. Up to sign, the numbers $b_i, 1 \leq i \leq n$ are the coefficients of the monic polynomial $q(t) \in \MZ[t]$ of degree $n$ with roots $a_i^2, 1 \leq i \leq n$. If all of them were divisible by an odd prime $p$, it would follow that $q(t) \equiv t^n \ (\text{mod }p)$. Since $t^n$ does not have $\pm 1$ or $\pm 2$ as roots in $\MZ/p\MZ$, this is a contradiction. Hence no odd prime divides all of the numbers $b_i$ and thus their greatest common divisor is a power of $2$.
\end{proof}

\section{Infinitely many smooth bundles over $BSU(2)$}\label{infinitely}

The purpose of this section is to prove that for a fixed manifold $W$, $G = SU(2)$ and $\alpha \in H^{\ast}(B\Diff(W);\MQ)$, it is possible that the map 
\begin{equation*}
D^{G,W}_{\alpha} \colon \Hom(G,\Diff(W))/ \Diff(W) \to H^{\ast}(BG;\MQ).  
\end{equation*}
from the introduction has an infinite image. As a consequence, we obtain that there are infinitely many different kinetic manifold bundles over $BSU(2)$ with the same fiber $W$. We will only consider the case $W = S^2 \times S^2$ which is sufficient to exhibit the described phenomenon.
\begin{Pro}
For any even integer $k \geq 0$, there exists a smooth action of $SU(2)$ on $S^2 \times S^2$ such that the induced map \emph{$s_k\colon BSU(2) \to B\Diff(S^2 \times S^2)$} satisfies $s_k^{\ast}\kappa_{ep_1} = 4(k^2 + 1)c_2$, where $c_2 \in H^4(BSU(2);\MZ)$ denotes the second Chern class. 
\end{Pro}
\begin{proof} Let $k \geq 0$ be even. Two-dimensional oriented real vector bundles over $S^2$ are classified by 
$\pi_2 BSO(2) \cong \MZ,$ and this number coincides with the Euler number of such a bundle. Consider the unique two-dimensional oriented real vector bundle $\xi$ over $S^2$ of Euler number $k$. If it is endowed with a Euclidean metric, the $S^1$-bundle of unit vectors can be identified with $SU(2) / (\MZ/k\MZ)$. It thus receives an action of $SU(2)$, via multiplication on the left. This $SU(2)$-action extends by scaling to the whole vector bundle, where the action on the zero section is given by left translation on $SU(2)/S^1 \cong S^2$. It is clear that this action is smooth near the zero section. Let us now consider the double of the unit disk bundle, which is a smooth $S^2$-bundle over $S^2$ and its total space $E$ also receives an action of $SU(2)$. It is classified by the composition $S^2 \to BSO(2) \to BSO(3),$ which is null since $k$ was assumed to be even, hence the bundle is trivial. In particular, its total space $E$ is diffeomorphic to $S^2 \times S^2$. Finally, for a maximal torus $S^1 \subset SU(2)$, the set of fixed points $E^{S^1}$ consists of $4$ points, and the weights of the normal representations at these points are $(\pm k, \pm 1)$. By Lemma \ref{pullbackformula}, this implies that the corresponding composite map $t\colon BS^1 \to BSU(2) \to B\Diff(S^2 \times S^2)$ satisfies $t^{\ast} \kappa_{ep_1} = 4(k^2 + 1)\gamma^{2} \in H^4(BS^1;\MQ).$ Since the isomorphism $H^4(BS^1; \MQ) \leftarrow H^4(BSU(2); \MQ)$ sends $c_2$ to $\gamma^2$, this implies the assertion.
\end{proof}

The preceding proposition is relevant for two different reasons. Firstly, it shows that the rather arithmetic nature of our main theorem is not just a mere coincidence---a result constraining $s^{\ast} \kappa_{c}$ for $s \colon BSU(2) \to B\Diff(W)$ induced by an action cannot just give a bound on all possible values that might arise. Secondly, it implies the following result.
\begin{Cor}\label{InfinitelyMany} There are infinitely many non-equivalent kinetic smooth $S^2 \times S^2$-bundles over $BSU(2)$.
\end{Cor}

\section{Non-kinetic manifold bundles over BSU(2)}

In this section we prove Theorem B. 
We first show that its assumptions are satisfied for the manifold $W_g$ of dimension $2n$ if $n \geq 3$ is odd and $g > 1$. The only part that needs to be explained  is the following.
\begin{Pro} \label{compliant}
For any $n \geq 3$ and $g \geq 1$, the manifold $W_g = \#^g(S^n \times S^n)$ admits a non-trivial smooth action of $SU(2)$.
\end{Pro}
\begin{proof} (Compare \cite[Section 4.1]{Ilya2}) Let $V = \mathbb R^3$ on which $SU(2)$ acts via the two-fold covering map $SU(2) \to SO(3)$ and the natural action of $SO(3)$ on $\Reals^3$. Furthermore, let $\Reals^{n-2}$ be endowed with a trivial action of $SU(2)$. Then $S^n$ can be identified with the unit sphere in $V \oplus \Reals^{n-2}$ and thus inherits a smooth action of $SU(2)$. This also gives a smooth action of $SU(2)$ on $S^n \times S^n$, by acting as described on the first factor and trivially on the second. The set of global fixed points of this action is $S^{n-3} \times S^n \neq \emptyset$. Now we can equivariantly glue $g$ copies together along fixed points to obtain a non-trivial smooth action of $SU(2)$ on $W_g$.
\end{proof}

To prove Theorom B, we recall a result originally due to Sullivan, see \cite[Corollary 5.10 on page 165]{Sullivan}, which was later refined by Mislin, see \cite{Mislin}.
\begin{Thm}[Sullivan-Mislin] \label{hpinfty} The map $[BSU(2),BSU(2)] \to \MZ$ given by looping a map and then taking the degree of the self-map of $\Omega BSU(2) \simeq S^3$ is injective and the image is $\{0\} \cup \{\text{odd squares}\}$. 
\end{Thm}
For $k \geq 1$ odd, the (up to homotopy uniquely defined) map $\psi \colon BSU(2) \to BSU(2)$ such that $\deg(\Omega \psi) = k^2$ is denoted by $\psi_k$. Sullivan calls these maps generalized Adams-operations; they relate to the more familiar Adams-operations in topological $K$-theory. It is clear that $\psi_k$ induces multiplication by $k^{2i}$ on $H^{4i}(BSU(2);\MQ) \cong \MQ$. We now prove Theorem B.

\begin{proof}
Let $W$ be as required by the assumptions. In particular, we assume that there exists a non-trivial smooth action of $SU(2)$ on $W$. From Theorem A, we know that the corresponding map $s\colon BSU(2) \to B\Diff(W)$ satisfies $s^{\ast} \kappa_{ep_i} = b_i\chi(W)c_2^{i},$ for integers $b_1, \dots, b_n$ whose greatest common divisor is a power of $2$. Now for any odd prime $p$, we can consider the corresponding Adams-operation $\psi_p\colon BSU(2) \to BSU(2)$. The composite map $f = s \circ \psi_p \colon BSU(2) \to B\Diff(W)$ then satisfies $f^{\ast} \kappa_{ep_i} = (\psi_p^{\ast} \circ s^{\ast}) (\kappa_{ep_i}) = p^{2i} s^{\ast}\kappa_{ep_i} = p^{2i}b_i\chi(W)c_2^{i},$ for $1 \leq i \leq n$. It now follows from Theorem A that $f$ is not induced by an action of $SU(2)$ on $W$, i.e., the smooth $W$-bundle over $BSU(2)$ corresponding to $f$ is non-kinetic.
\end{proof}

\section{Some open questions}\label{Open}
One might wish to go beyond Theorem B by changing the group or the manifold under consideration. The circle is often considered to be the most fundamental Lie group, so it would be interesting to answer the following question. 

\begin{Que} Does there exist a non-kinetic smooth manifold bundle over $BS^1$? If so, for which smooth manifolds $W$ as fibers is this possible?
\end{Que}

The author will refrain from speculating whether such bundles exist. However, the following proposition provides a negative answer for surfaces.

\begin{Pro} For a closed connected oriented surface $\Sigma$, every smooth $\Sigma$-bundle over $BS^1$ is kinetic.
\end{Pro}
\begin{proof} 
Smale proved that the inclusion $SO(3) \hookrightarrow \Diff(S^2)$ is a weak equivalence \cite{Smale}, and Dwyer--Zabrodsky showed that for any compact Lie group $H$ and any $p$-toral group $P$, the map $\Hom(P,G)/G \to [BP,BG]$ is bijective, see \cite{DwyerZabrodsky, Notbohm}. When both results are combined, it follows that all bundles $BS^1 \to B\Diff(S^2)$ are indeed kinetic, which settles the case $\Sigma \cong S^2$. In the case $\Sigma \cong T^2$, it follows from Smale's theorem that there is a homotopy fiber sequence $BT^2 \to B\Diff(T^2) \to B\GL_2(\MZ),$
see \cite{EE67}. By covering space theory, the map $BS^1 \to B\Diff(T^2)$ is homotopic to one that can be lifted to a map $BS^1 \to BT^2$. Now the conclusion follows again from the result of Dwyer-Zabrodsky. In the remaining case of higher genus, the Earle-Eells theorem states that the components of $\Diff(\Sigma _g)$ are contractible \cite{EE67}. Hence the space $B\Diff(S_g)$ is a $K(\Gamma_g,1)$, with $\Gamma_g = \pi_0 \Diff(\Sigma _g)$ the mapping class group. Thus, again by covering theory, every map $BS^1 \to B\Diff(S_g)$ is null. In particular, every surface bundle over $BS^1$ with fiber a surface of higher genus is trivial.
\end{proof}

Going in a different direction, one can also keep $G = SU(2)$, but enlarge the class of manifolds that are considered. The following conjecture seems plausible.

\begin{Con} Let $W$ be an arbitrary closed manifold and suppose $SU(2)$ acts smoothly and non-trivially on it. Let $s\colon BSU(2) \to B\Diff(W)$ be the induced map on classifying spaces. For $k > 1$ odd, the bundle classified by $s \circ \psi_k \colon BSU(2) \to B\Diff(W)$ is non-kinetic.
\end{Con}

Theorem B says that this conjecture is true for many manifolds, in particular it holds for $\#^g (S^n \times S^n)$ in case $n \geq 3$ is odd and $g > 1$. It is unclear whether it also holds for such manifolds if $n$ is even.

\end{document}